\documentclass[a4paper,11pt]{article}

\usepackage[headinclude]{typearea}
\usepackage[english]{babel}           %
\usepackage[utf8]{inputenc}
\usepackage[T1]{fontenc}
\usepackage[utf8]{inputenc}
\usepackage{natbib}
\usepackage{amsthm,amsmath, amssymb,amsfonts}

\theoremstyle{plain}
\newtheorem{theorem}{Theorem}[section]

\newtheorem{lemma}[theorem]{Lemma}

\theoremstyle{definition}
\newtheorem{assumption}[theorem]{Assumption}
\newtheorem{definition}[theorem]{Definition}
\newtheorem{example}[theorem]{Example}
\newtheorem{remark}[theorem]{Remark}

\newcommand{\E}{{\mathbb{E}}}
\newcommand{\N}{{\mathbb{N}}}
\renewcommand{\P}{{\mathbb{P}}}
\newcommand{\R}{{\mathbb{R}}}
\newcommand{\tr}{\mathrm{tr}}

\newcommand{\cD}{{\cal D}}
\newcommand{\cF}{{\cal F}}

\newcommand{\cL}{{\cal L}}
\newcommand{\domain}{\R^d}

\newcommand{\usol}{unique local solution}
\newcommand{\ugsol}{unique global solution}
\newcommand{\msol}{unique maximal local solution}

\begin{document}

\title{
On the Existence of Solutions of a Class of SDEs with Discontinuous Drift and Singular Diffusion
\thanks{G. Leobacher and M. Sz\"olgyenyi are supported by the Austrian Science
Fund (FWF): Project F5508-N26, which is part of the Special Research Program
"Quasi-Monte Carlo Methods: Theory and Applications". \newline M. Sz\"olgyenyi
is supported by the Austrian Science Fund (FWF) Project P21943. \newline S.
Thonhauser is supported by the Swiss National Science Foundation (SNF) Project
200021-124635/1, by an Upper Austrian excellence scholarship, and by the
Austrian Science Fund (FWF): Project F5510-N26, which is part of the Special
Research Program "Quasi-Monte Carlo Methods: Theory and Applications".}} 

\author{%
  Gunther Leobacher\footnote{Johannes Kepler University, Austria.
    \tt{gunther.leobacher@jku.at}}
  ~and 
  Stefan Thonhauser\footnote{Technical University Graz, Austria. \tt{stefan.thonhauser@math.tugraz.at} }
  ~and 
  Michaela Sz\"olgyenyi\footnote{Johannes Kepler University, Austria.
    \tt{michaela.szoelgyenyi@jku.at}}
}

\date{March 2021}

\maketitle
\begin{abstract}
The classical result by It\^o on the existence of strong solutions 
of stochastic differential equations (SDEs) with Lipschitz coefficients 
can be extended to the case where the drift is only measurable and bounded.
These generalizations are based on techniques presented by \citet{zvonkin1974}
and \citet{veretennikov1981}, which rely on the uniform ellipticity of the
diffusion coefficient. 

In this paper we study the case of degenerate ellipticity and give sufficient
conditions for the existence of a solution. The conditions on the diffusion
coefficient are more general than previous results and we gain fundamental 
insight into the geometric properties of the discontinuity of the drift on the
one hand and the diffusion vector field on the other hand. \\
Besides presenting existence results for the degenerate elliptic situation, we
give an example illustrating the difficulties in obtaining more general results
than those given.

The particular types of SDEs
considered arise naturally in the framework of combined optimal control and
filtering problems.\\

\end{abstract}

\section{Introduction}
\label{sec:Introduction}

In this article we are going to consider a $d$-dimensional time-homogeneous
stochastic differential equation (SDE) of the form 
\begin{align}\label{eq:SDEintro}
dX_t &= \mu (X_t) \, dt + \sigma (X_t) \,dW_t \,.
\end{align}
It\^o's well-known existence and uniqueness theorem for SDEs states that for
(locally) Lipschitz $\mu,\sigma$ there exists a unique (maximal local)
strong solution, see, e.g., \cite{mao2007}. However, in this article, we are
interested in the case where $\mu$ is not Lipschitz.

\citet{zvonkin1974} (for the one-dimensional case) and \citet{veretennikov1981}
(for the multi-dimensional case) prove that if
$\sigma$ is bounded and Lipschitz and the
infinitesimal generator of the SDE is {\em uniformly elliptic}, i.e., if there
exists a constant $\lambda>0$ such that for all $x\in \domain$ and all $v\in \domain$ we
have $v^\top\sigma(x)\sigma(x)^\top v\ge \lambda v^\top v$, then there
still exists a strong solution, even if $\mu$ is only measurable and bounded.
\citet{veretennikov1984} generalizes the result  by requiring
that uniform ellipticity needs to hold only for those components in which the drift is non-Lipschitz.

Zvonkin's method is extended by \citet{zhang2005} to a locally integrable
drift function and non-degenerate diffusion.  Beyond the aforementioned
classical results one can find several approaches for dealing with
discontinuous drift coefficients in the literature. 
A very natural way is to use smooth approximations for discontinuous coefficients and analyze the corresponding limiting process.
Such a procedure is presented in a general form by \citet{krylov2002}, but still the presence of ellipticity is crucial.
This technique is also related to the question of stability of SDEs, see \citet[Chapter V.4]{protter2004},
for which results commonly ask for the a-priori existence of the limiting process.
In \citet{meyer2010} the existence question for the situation of a measurable drift function
and non-degenerate diffusion coefficient is dealt with using techniques from Malliavin calculus.

Another method, different in spirit, is introduced by \citet{halidias2006} who prove existence of a solution to an SDE with a drift
that is increasing in every coordinate via a construction using sub- and super
solutions. For our result, no monotonicity of the drift is needed and it is 
therefore a more  general existence and uniqueness result for the
degenerate setup.
\\

That ellipticity plays a crucial role can be illustrated by the following example for which 
uniform ellipticity fails: consider the 2-dimensional SDE
\begin{equation}\label{eq:sde_no_sol}
\begin{aligned}
dX^1_t&=\left(\frac{1}{2}-\mathrm{sgn}(X^1_t+X^2_t)\right)\, dt+\,dW_t\,, \\
dX^2_t&=-dW_t\,,
\end{aligned}
\end{equation}
with $X_0=(0,0)$. If \eqref{eq:sde_no_sol} had a strong solution then we see
from adding the two equations that there would exist a one-dimensional adapted
process $\tilde X=X^1+X^2$ satisfying
\begin{equation}\label{eq:sde_no_sol1}
\tilde X_t=\int_0^t\left(\frac{1}{2}-\mathrm{sgn}(\tilde X_s)\right) ds\,.
\end{equation}
For the sake of completeness we show in Section \ref{sec:Example} that such a process $\tilde X$ cannot exist.\\

The aim of this article is to give sufficient conditions on $\mu,\sigma$ such
that the SDE has a solution  for the case where $\sigma\sigma^\top$ fails to be
uniformly elliptic and where $\mu$ is allowed to be discontinuous and unbounded. This question is motivated by an example from \citet{sz12},
where a
SDE appears that clearly violates the classical conditions of the theorem by It\^o
and does not necessarily include an increasing drift like the example in \citet{halidias2006}.
In fact, we are especially interested in cases where the drift is decreasing. We will
present the example from \citet{sz12} as an application of our results in
Section \ref{sec:Example}.\\ 

The first main result of our contribution is Theorem \ref{thm:theorem1} which
states that there exists a \msol{} of \eqref{eq:SDEintro}, if $\sigma$
is sufficiently smooth, $(\sigma \sigma^\top)_{11}\ge c >0$, and if $\mu$ is
discontinuous in $\{x_1=0\}$, but sufficiently smooth everywhere else.  The
result can be generalized by a transformation of the domain, to allow for a
discontinuity along a hypersurface, but what remains is a certain
dependence of the result on the geometry of the discontinuity of the drift
coefficient.\\

The contribution of this article is two-fold: firstly, it closes a gap in 
combined filtering-control problems, namely that of admissibility of 
threshold and band strategies, which are very common types of optimal 
strategies.
Secondly, it contributes to the theory of SDEs, not only by giving one of the
most general existence and uniqueness results for the degenerate-elliptic
case, but also by highlighting intriguing connections between the geometry of a 
discontinuity of the drift on the one hand and the diffusion vector field
on the other hand. \\

The paper is organized as follows. In Section \ref{sec:First} we first fix
notations and present the classical notions of strong, local, and maximal local
solutions. Towards the statement of the main theorem a technical lemma and a
particular form of It\^o's formula are proven.  In Section \ref{sec:Result} we
prove the main results of this paper, and in Section \ref{sec:Example} we apply
our results to a concrete problem coming from mathematical
finance. 

\section{Definitions and first results}
\label{sec:First}

In the whole paper we work with a filtered probability space $(\Omega, \cF, (\cF_t)_{t \ge 0}, \P)$, where the filtration satisfies the usual conditions.
Furthermore, we consider a $d$-dimensional standard Brownian motion $W=(W_t)_{t\ge 0}$ on that space.
By $\Vert \cdot \Vert$ we always denote the $d$-dimensional Euclidean norm.
When using the notion ``solution'' we always refer to ``strong solution''.\\

First, let us recall the definitions of local, maximal local, and global solutions of SDEs. Consider an SDE of the form
\begin{align}\label{eq:defSDE}
 X_t = x_0 + \int_0^t \mu(X_s)\,ds + \int_0^t \sigma(X_s)\,dW_s\,,
\end{align}
with initial data $X_0 = x_0 \in \cL_{\cF_{0}}^2 (\Omega)$, where
\begin{align*}
\mu : \domain \rightarrow \domain \, \, \mbox{ and } \, \, \sigma: \domain \rightarrow \R^{d \times d}\,.
\end{align*}

\begin{definition}[\mbox{\cite{veretennikov1981}}]
An $\domain$-valued stochastic process $(X_t)_{0 \le t \le T}$ is called a \textit{solution} of \eqref{eq:defSDE}, if it is continuous,
progressively measurable w.r.t. $(\cF_t)_{t\ge0}$, and if it fulfills \eqref{eq:defSDE} for all $t\in[0, T]$ a.s.\\
A solution $(X_t)_{0 \le t \le T}$ is said to be {\it unique}, if any other
solution $(\bar{X}_t)_{0 \le t \le T}$ is indistinguishable from $(X_t)_{0 \le
t \le T}$.  \end{definition} 

\begin{definition}[\mbox{\cite[Definition 3.14]{mao2006}}]
Let $\zeta$ be an $(\cF_t)_{t\ge0}$-stopping time such that $0 \le \zeta \le T$ a.s. An
$\domain$-valued $(\cF_t)_{t\ge0}$-adapted continuous stochastic process $(X_t)_{0 \le t <
\zeta}$ is called a {\it local solution} of equation \eqref{eq:defSDE}, if
$X_0=x_0$, and,
moreover, there is a non-decreasing sequence $\{\zeta_k\}_{k \ge 1}$ of $(\cF_t)_{t\ge0}$-stopping times such that $0 \le \zeta_k \uparrow \zeta$ a.s. and 
\begin{align*}
 X_t = x_0 + \int_{0}^{t \wedge \zeta_k} \mu(X_s) \,ds + \int_{t_0}^{t \wedge \zeta_k} \sigma(X_s)\, dW_s
\end{align*}
holds for any $t \in [0, T)$ a.s.\\
If, furthermore, 
\begin{align*}
 \limsup_{t \to \zeta} \Vert X_t \Vert = \infty \, \,  \mbox{ whenever } \, \, \zeta < T\,,
\end{align*}
then it is called a {\it maximal local solution} and $\zeta$ is called the {\it explosion time}.\\
A maximal local solution $(X_t)_{0 \le t < \zeta}$ is said to be {\it unique}, if any other maximal local solution $(\bar{X}_t)_{0 \le t < \bar{\zeta}}$ is indistinguishable from it,
namely $\zeta = \bar{\zeta}$ and $X_t=\bar{X}_t$ for $0 \le t < \zeta$ a.s.
\end{definition}

\begin{definition}[\mbox{\cite[p. 94]{mao2006}}]
If the assumptions of an existence and uniqueness theorem hold on every finite subinterval $[0, T]$ of $[0, \infty)$, then \eqref{eq:defSDE} has a unique solution
$(X_t)_{t\ge0}$ on the entire interval $[0, \infty)$. Such a solution is called a {\it global} solution.
\end{definition}

Furthermore, recall that a function $\phi$ on $\domain$ is locally Lipschitz, if for all $n \in \N$ there is a constant $L_n>0$ such
that for those $x_1,x_2 \in \domain$ with $\max \{ \Vert x_1 \Vert, \Vert x_2 \Vert \} \le n$ the following condition holds:
\begin{align*}
 \Vert \phi(x_1)-\phi(x_2) \Vert \le L_n \Vert x_1-x_2\Vert \,.
\end{align*}

Consider the following system of SDEs on $\domain$:

\begin{align}\label{eq:SDE}
dX_t &= \mu (X_t) d t + \sigma (X_t) \,dW_t\,,
\end{align}

$X_0=x$.

\begin{assumption}\label{ass:terms}
We assume the following for the coefficients of \eqref{eq:SDE}.
 \begin{enumerate}
  \item[(i)] $\sigma_{ij}$, for $i=2,\ldots,d, j=1,\ldots,d$ are locally Lipschitz,
  \item[(ii)] $\sigma_{1j} \in C^{1,3}(\R\times\R^{d-1})$, for $j=1,\ldots,d$,
  \item[(iii)] $(\sigma \sigma^\top)_{11}\ge c > 0$ for some constant $c$ and for all $x \in \domain$.
 \end{enumerate}
\end{assumption}

The function $\mu:\domain\longrightarrow\domain$ is allowed to be discontinuous. However, the form of the
discontinuity needs to be a special one. 

\begin{assumption}\label{ass:b1}
We assume the following for $\mu$: there exist functions
$\mu^+,\mu^-\in C^{1,3}(\R\times\R^{d-1})$ such
that
\[
\mu(x_1,\ldots,x_d)=\left\{
\begin{array}{ccc}
\mu^+(x_1,x_2,\ldots,x_d)&\text{ if }&x_1>0\\
\mu^-(x_1,x_2,\ldots,x_d)&\text{ if }&x_1<0\\
\end{array}
\right.
\]
\end{assumption}
\begin{remark}
The value of $\mu$ for $x_1=0$ is of no significance for us since,
due to Assumption \ref{ass:terms} (iii) a solution $X$ to \eqref{eq:SDE}
does not spend a positive amount of time in the hyperplane $\{x_1=0\}$.
\end{remark}
Now, we are going to study the existence of a solution to system \eqref{eq:SDE}.
Suppose first that a solution to \eqref{eq:SDE} exists and define
\[
Z^1_t = g_1 (X_t)
\]
for some suitable function $g_1:\domain\longrightarrow\R$. For the sake of readability we are going to skip the arguments in the following.\\
Heuristically using It\^o's formula we calculate
\begin{equation*}
\begin{aligned}
dZ^1_t &=\sum_{i=1}^d \frac{\partial g_1}{\partial x_i} dX^i_t + \frac{1}{2}\sum_{i,j=1}^d \frac{\partial^2 g_1}{\partial x_i\partial x_j} d[X^i,X^j]_t\\
&= \left[\mu_1 \frac{\partial g_1}{\partial x_1} + \frac{1}{2} (\sigma \sigma^\top)_{11}
 \frac{\partial^2 g_1}{\partial x_1^2} \right] d t + \sum_{i=2}^d \mu_i \frac{\partial g_1}{\partial x_i} dt
 \\ &+ \frac{1}{2} \sum_{i,j=1}^d(1-\delta_{11}(i,j)) (\sigma \sigma^\top)_{ij} \frac{\partial^2 g_1}{\partial x_i \partial x_j} dt +  \sum_{i,j=1}^d  \sigma_{ij} \frac{\partial g_1}{\partial x_i} dW_t^j\,.
\end{aligned}
\end{equation*}
Now we would like to choose $g_1 \not\equiv 0$ such that
\begin{align}\label{eq:g}
\mu_1 \frac{\partial g_1}{\partial x_1} + \frac{1}{2} (\sigma \sigma^\top)_{11}
 \frac{\partial^2 g_1}{(\partial x_1)^2} = 0
\end{align}
to eliminate the problematic term $\mu_1$. This extends the idea of \citet{zvonkin1974}. 
A solution $g_1$ can be obtained as follows.
\begin {align*}
&\frac{\frac{\partial ^2 g_1}{(\partial x_1)^2}} {\frac{\partial g_1}{ \partial x_1}} 
= - \frac{2 \mu_1}{(\sigma \sigma^\top)_{11}} 
= \frac{\partial}{\partial x_1} \left(\log \left(\frac{\partial g_1}{\partial x_1}\right)\right)\\
& \Leftrightarrow g_1 = C_0(x_2,\dots,x_d) \int \exp \left(- \int \frac{2 \mu_1}{(\sigma \sigma^\top)_{11}} d x_1 \right) dx_1 + C_1 (x_2,\dots,x_d)\,.
\end {align*}
We are free to choose $C_1 (x_2,\dots,x_d) \equiv 0$ and $C_0 (x_2,\dots,x_d) \equiv 1$, such that
\[
g_1(x) = \int_0^{x_1} \exp \left(- \int_0^\xi \frac{2 \mu_1(t,x_2,\dots,x_d)}{(\sigma \sigma^\top)_{11}(t,x_2,\dots,x_d)}\, dt \right) d\xi \,.
\]
Note that $\frac{\partial g_1}{\partial x_1} (0,x_2,\ldots,x_d)> 0$ for all $(x_2,\ldots,x_d)\in\R^{d-1}$.
Note further that for $i\ne 1$ the terms of the 
form $\mu_i \frac{\partial g_1}{\partial x_i}$ are locally Lipschitz since
$\mu_i$ is locally bounded and $\frac{\partial g_1}{\partial x_i}$ is zero on 
$\{x_1=0\}$.

Using similar considerations we can find $g_k:\domain\longrightarrow\R$ such that for $Z^k=X^k+g_k(X)$, $k=2,\ldots,d$, the
drift coefficient is Lipschitz. The corresponding ordinary differential equation is 
\begin{align}\label{eq:gk}
\mu_k+\mu_1 \frac{\partial g_k}{\partial x_1} + \frac{1}{2} (\sigma \sigma^\top)_{11}
 \frac{\partial^2 g_k}{(\partial x_1)^2} = 0
\end{align}
for which we get a special solution 
\[
g_k(x_1,\ldots,x_d) := \int_0^{x_1} C_k(\xi,x_2,\ldots,x_d)\exp \left(- \int_0^\xi
\frac{2\mu_1(t,x_2,\ldots,x_d)}{(\sigma \sigma^\top)_{11}(t,x_2,\ldots,x_d)} \,dt \right) d\xi \,,
\]
where 
\[
C_k(\xi,x_2,\ldots,x_d):=
-\int_0^{\xi} \frac{2\mu_k(\eta,x_2,\ldots,x_d)}{(\sigma
\sigma^\top)_{11}(\eta,x_2,\ldots,x_d)}\exp \left(\int_0^\eta
\frac{2\mu_1(t,x_2,\ldots,x_d)}{(\sigma
\sigma^\top)_{11}(t,x_2,\ldots,x_d)} \,dt \right) d\eta 
\]
by the method of variation of constants. Note that for all $k\ne 1$ and all $i=1,\ldots,d$ we have
$\frac{\partial g_k}{\partial x_i}(0,x_2,\ldots,x_d) = 0$ for all $(x_2,\ldots,x_d)\in\R^{d-1}$.

\begin{lemma}\label{lem:lemma1}
Consider a function $g:\domain\longrightarrow\R$ of the form
\[
g(x_1,\ldots,x_d) = \int_0^{x_1} C(\xi,x_2,\ldots,x_d)\exp \left( \int_0^\xi a(t,x_2,\ldots,x_d)\,dt \right) d\xi\,,
\]
where 
$C$ is of the form 
\[
C(\xi,x_2,\ldots,x_d)=\int_0^\xi c_1(\eta,x_2,\ldots,x_d)d\eta+c_0
\]
for some $c_0\in \R$ and $a, c_1$ are such that 
\begin{align*}
a(x_1,\ldots,x_d)&=\left\{
\begin{array}{ccc}
a^+(x_1,x_2,\ldots,x_d)&\text{ if }&x_1>0\\
a^-(x_1,x_2,\ldots,x_d)&\text{ if }&x_1<0\\
\end{array}
\right.\\
c_1(x_1,\ldots,x_d)&=\left\{
\begin{array}{ccc}
c_1^+(x_1,x_2,\ldots,x_d)&\text{ if }&x_1>0\\
c_1^-(x_1,x_2,\ldots,x_d)&\text{ if }&x_1<0\\
\end{array}
\right.
\end{align*}
for some functions $a^+,a^-,c_1^+,c_1^-\in C^{1,3}(\R\times\R^{d-1})$.

Then $Dg$  is locally Lipschitz for $D\in \{1,\frac{\partial}{\partial x_i},\frac{\partial^2}{\partial x_i\partial x_j}\big| i,j=1,\ldots,d\}\backslash\{\frac{\partial^2}{\partial x_1^2}\}$.
\end{lemma}

\begin{proof}
For $i\ne 1$ 
we may exchange differentiation w.r.t. $x_i$ and integration (see, e.g., 
\cite[Theorem 9.42]{rudin1976}).
Furthermore, we use
that  $a,c_1$ have continuous third derivatives w.r.t. $(x_2,\ldots,x_d)$, as warranted by our assumptions on the functions $a,c_1$.
Thus we get that $g,\frac{\partial g}{\partial x_i}$ for $i=1,\ldots,d$ are locally Lipschitz and $\frac{\partial^2 g}{\partial x_i\partial x_j}$ is locally Lipschitz for $i,j\ne 1$.

It remains to show that $\frac{\partial^2 g}{\partial x_j\partial x_1}$ is locally Lipschitz for all $j>1$. 
Let $r>0$ and let
$x,y\in B_r(0)$. 
$\frac{\partial^3 g}{\partial x_j^2\partial x_1}$,
$\frac{\partial^3 g}{\partial x_j\partial x_1^2}$ exist and are bounded on $B_r(0)$, say by
$K$.
Consider first the case where $x_1,y_1>0$ (i.e., both points lie on the
same side of the hyperplane $\{x_1=0\}$):
\begin{align*}
\left|\frac{\partial^2 g}{\partial x_j\partial x_1}(y)-
\frac{\partial^2 g}{\partial x_j\partial x_1}(x)\right|
\le \sup_{\Vert z\Vert \le r}\left\|\left(\frac{\partial^3 g}{\partial x_j^2\partial x_1}(z),
\frac{\partial^3 g}{\partial x_j\partial x_1^2}(z)\right)\right\|\,\|y-x\|
\le\sqrt{2}K \|y-x\|\,.
\end{align*}
If exactly one of the $x_1,y_1$ is zero, the same estimate holds.

If $x_1<0<y_1$, let $z$ be the intersection of the hyperplane $\{x_1=0\}$
with the line connecting $x$ and $y$. Now make the same estimate as above twice. 

Finally, let $x_1=y_1=0$. Let $z_1=\|y-x\|/2$, $(z_2,\ldots,z_d)=\frac{1}{2}((x_2,\ldots,x_d)+(y_2,\ldots,y_d))$. Then
\begin{align*}
\left|\frac{\partial^2 g}{\partial x_j\partial x_1}(y)-
\frac{\partial^2 g}{\partial x_j\partial x_1}(x)\right|&\le
\left|\frac{\partial^2 g}{\partial x_j\partial x_1}(y)-
\frac{\partial^2 g}{\partial x_j\partial x_1}(z)\right|
+\left|\frac{\partial^2 g}{\partial x_j\partial x_1}(z)-
\frac{\partial^2 g}{\partial x_j\partial x_1}(x)\right|\\
&\le \sqrt{2}K(\|y-z\|+\|z-x\|) =
2K\|y-z\|\,.
\end{align*}
\end{proof}

From Lemma \ref{lem:lemma1} together with Assumptions \ref{ass:terms} (ii) and \ref{ass:b1} it follows that for all $k=1,\ldots ,d$ the function $D g_k$ is locally Lipschitz for $D\in \{1,\frac{\partial}{\partial x_i},\frac{\partial^2}{\partial x_i\partial x_j}\big| i,j=1,\ldots,d\}\backslash\{\frac{\partial^2}{\partial x_1^2}\}$.

Define a function $G:\domain\longrightarrow\domain$ by
\begin{align*}
G(x_1,\ldots,x_d)&:=(g_1(x),x_2+g_2(x),\ldots,x_d+g_d(x))\,.
\end{align*}
Then $\frac{\partial G_k}{\partial x_i}(0,x_2,\ldots,x_d)=\delta_{i,k}$.
In particular, due to the inverse function theorem \cite[Theorem 9.24]{rudin1976}, $G$ is locally invertible, i.e., 
for every $x_0\in\domain$ there exist $r>0$ and $H:G(B_{r}(x_0))\to\domain$ such that
\begin{align*}
H\circ G=id_{B_r(x_0)}\mbox{ and } G\circ H=id_{G(B_r(x_0))}\,.
\end{align*}
Note that $H$ inherits the smoothness from $G$, see \cite[Theorem A.1]{leobacher2020}.\\

Before proceeding we need to clarify the notion \emph{local} in our context. Below we are going to use the locally defined function $H$ for establishing the existence of a solution to
\eqref{eq:SDE}. Naturally, for $X_0=x_0$ the -- still to be constructed -- solution $X$ exists on $B_r(x_0)$.
Setting $\zeta^k=\zeta=\inf\{t>0\,\vert\,X_t\notin B_r(x_0)\}$ for $k\in\N$ in Definition 2.2 we will have that $(X_{t\wedge \zeta^k})_{t\geq 0}$
fulfills \eqref{eq:SDE}. Therefore it is a local solution in the sense of Definition 2.2 but on the restricted domain $B_r(x_0)$
we will have existence of a \emph{unique} solution to \eqref{eq:SDE} for every initial point $X_0\in B_r(x_0)$.
It is maximal in the sense that there will be no explosion before reaching the boundary of $B_r(x_0)$.\\

Now, let us assume for the moment that a solution
$X$ to the SDE \eqref{eq:SDE} exists.
Let $Z_t:=G(X_t)$. The functions $g_k$ for 
$k=2,\ldots,d$ are defined in a way to guarantee that the drift of 
$Z^k$ is locally Lipschitz, i.e., the discontinuities are removed from
the drift.\\

Now, let us consider the following 'transformed' SDE:
\[
dZ_t=\left(\nabla G(X_t)\mu(X_t) 
+ \tr\left(\sigma(X_t)^\top \nabla^2 G(X_t)\sigma(X_t)\right)\right)\,dt
+\nabla G(X_t)\sigma(X_t)dW_t
\]
That is, $X$ solves $dX_t=\mu(X_t)dt+\sigma(X_t) dW_t$ iff 
$Z_t=G(X_t)$ solves 
\begin{equation}\label{eq:newsystem}
\left.
\begin{aligned}
dZ_t&=\tilde\mu(Z_t)dt+\tilde\sigma(Z_t)dW_t\qquad \text{ with}\\
\tilde\mu(z)&=(\nabla G)(H(z))\mu(H(z)) 
+ \tr\left(\sigma(H(z))^\top (\nabla^2 G)(H(z))\sigma(H(z))\right)\,,\\
\tilde \sigma(z)&=(\nabla G)(H(z))\sigma(H(z))\,.
\end{aligned}
\right\}
\end{equation}

\begin{lemma}\label{thm:lemma2}
Let Assumptions \ref{ass:terms} and \ref{ass:b1} be fulfilled.\\
Then system \eqref{eq:newsystem} has a \usol{} $Z$.
\end{lemma}
\begin{proof}
The drift and diffusion coefficients of \eqref{eq:newsystem} are locally Lipschitz by Lemma \ref{lem:lemma1}.
Thus, from \cite[Theorem 3.15]{mao2006} we get that \eqref{eq:newsystem} has a \usol{} $Z$.
\end{proof}
For proving our main result we need an It\^o-type formula for the function $H$
and the solution $Z$. There is a great number of extensions of the 
classical It\^o formula for functions that are not necessarily $C^2$, e.g.,
\citet{russo1996, foellmer2000, eisenbaum2000}, 
but usually they rely on non-degeneracy of the
diffusion coefficient of the
argument process, either explicitly, or implicitly, by having 
Brownian motion
as the argument.

For $x\in \R^d$ and $r\in (0,\infty)$ we denote by $B_r(x)$ the open ball with 
center $x$ and radius $r$. 

\begin{theorem}[It\^o's formula]\label{thm:ito}
For every $i\in\{1,\dots,n\}$, denote 
$\Theta^i\:=\big\{(x_1,\ldots,x_n)\colon x_i=0\big\}\subseteq\R^d$ and let
$f:\R^d\to\R$ be a function  such that
\begin{enumerate} 
\item $f$ is continuously differentiable, 
\item for all $i,j\in \{1,\ldots,d\}$ with $i\ne j$ and all $x\in\R^d$, 
the  mixed partial derivative 
$\frac{\partial^2 }{\partial x_i \partial x_j}f(x)$ 
exists,  and $\frac{\partial^2 }{\partial x_i \partial x_j}f\colon \R^d\to \R$ is continuous,
\item for all $i\in \{1,\ldots,d\}$ and all $x\in\R^d\setminus \Theta^i$, the  second partial derivative
$\frac{\partial^2 }{\partial x_i^2}f(x)$ exists, and  
$\frac{\partial^2 }{\partial x_i^2} f\colon \R^d\setminus\Theta^i\to \R$ is 
continuous,
\item for all $R\in[0,\infty)$ and all $i\in \{1,\ldots,d\}$, 
$\frac{\partial^2 }{\partial x_i^2} f\colon \R^d\setminus\Theta^i \to \R$ 
is bounded on $B_R(0)\setminus \Theta\,$.
\end{enumerate}

Furthermore, let $X$ be a continuous semimartingale starting in $x_0\in\R^d\,$.
Then for  all $t \in [0,\infty)$,
 \begin{align*}
f(X_t)&=f(x_0)+\sum_{i=1}^d \int_0^{t} \frac{\partial }{\partial x_i}f(X_s)\,dX_s^i
+\frac{1}{2} \sum_{i,j=1}^d \int_0^{t} \frac{\partial^2 }{\partial x_i\partial x_j}f(X_s)\,d[X^i,X^j]_s\,.
\end{align*}
\end{theorem}

We defer the proof to Appendix \ref{sec:app-ito}.
\section{Main result}
\label{sec:Result}

\begin{theorem}\label{thm:theorem1}
Let Assumptions \ref{ass:terms} and \ref{ass:b1} be fulfilled. Let $Z$ be the 
\usol{} of \eqref{eq:newsystem}.
Then
\begin{align*}
X=H(Z) 
\end{align*}
is the \usol{} to \eqref{eq:SDE}.
\end{theorem}

\begin{proof}
We have that $\nabla G(0,x_2,\ldots,x_d)$ is the $(d\times d)$-identity matrix, and if $x_1=0$, then also $z_1=0$.
Furthermore, the components of $H$ fulfill the assumptions of Theorem \ref{thm:ito}. Therefore,
It\^o's formula still holds.
Its application to $X=H(Z)$ yields
\[
dX_t =  \mu(X_t) \, dt + \sigma(X_t) d W_t\,,
\]
which proves the claim.
\end{proof}

From Theorem \ref{thm:theorem1} we know that a \usol{} of system \eqref{eq:SDE} exists. Now it remains to prove that there is a \msol.

\begin{theorem}\label{thm:theorem2}
Let Assumptions \ref{ass:terms} and \ref{ass:b1} be fulfilled.
Then system \eqref{eq:SDE} has a \msol.
\end{theorem}

\begin{proof}
 The proof consists of three steps.\\
 
 \underline{Step 1:} For each $x \in \domain$ there is a ball $B_{\varepsilon_x}(x)$ with radius $\varepsilon_x>0$
such that \eqref{eq:SDE} with $X_0=x$ has a \usol{} due to Theorem \ref{thm:theorem1}.\\
 
 \underline{Step 2:} Let $\cD_n$, $n\in \N$ be compact subsets of $\domain$ with
$\cD_n \uparrow \domain$, i.e., $\domain=\bigcup_{n \in \N} \cD_n$.
Furthermore, for all $x \in \domain$ there is an $n_0 \in \N$ such that $x \in \cD_n$ for all $n\ge n_0$.
\\
 We know from above that for all $x \in \cD_n$ there is a radius $\varepsilon_x>0$ such that \eqref{eq:SDE}
with $X_0=x$ has a \usol{} on $B_{\varepsilon_x}(x)$. Clearly, $\cD_n \subseteq \bigcup_{x \in \cD_n} B_{\varepsilon_x}(x)$
and since $\cD_n$ is compact, there exists $m<\infty$ such that $\cD_n \subseteq \bigcup_{k=1}^m B_{\varepsilon_{x_k}}\!\!(x_k)$, i.e., the covering is finite.\\
 Now, consider some fixed $x \in \cD_n$, $x \in B_{\varepsilon_{x_{k_1}}}\!\!(x_{k_1}), \dots, B_{\varepsilon_{x_{k_{\bar{m}}}}}\!\!(x_{k_{\bar{m}}})$
for $\{k_1, \dots, k_{\bar{m}}\} \subseteq \{1,\dots,m\}$ and $\bar{m}\le m$.
Since we have uniqueness of the solution on every ball, we have uniqueness on any finite intersection $\bigcap_j B_{\varepsilon_{x_{k_j}}}\!\!(x_{k_j})$.
 Thus, \eqref{eq:SDE} with $X_0=x$ has a \usol{} $(\Theta^x,\zeta^x)$ on $\bigcap_{j=1}^{\bar{m}} B_{\varepsilon_{x_{k_j}}}\!\!(x_{k_j})$, where $\Theta=X$.\\
 In detail, this means that there exist mappings
 \begin{align*}
  & \zeta: \Omega \times \cD_n \to [0,\infty]\,,\\
  & \Theta: \Omega \times \cD_n \times [0,\zeta) \to \cD_n\,,
 \end{align*}
 such that
 \begin{itemize}
  \item $\Theta$ and $\zeta$ are measurable (cf. \cite[Chapter V.6, Theorem 31]{protter2004} and due to a continuous transformation of a measurable function),
  \item $\zeta^x: \Omega \to [0,\infty]$, $\zeta^x=\inf\{ t > 0\vert \Theta_t^x \notin \bigcap_{j=1}^{\bar{m}} B_{\varepsilon_{x_{k_j}}}\!\!(x_{k_j}) \}$ is a stopping time w.r.t. our filtration, $\zeta^x>0$ a.s.
  \item $\Theta^x:\Omega \times [0,\zeta^x] \to \cD_n$ locally solves \eqref{eq:SDE} with $X_0=x$.
 \end{itemize}
 We can now use this to construct a solution on $\cD_n$. Let $\omega \in \Omega$. Define 
 \begin{align*}
  & \zeta_1(\omega):=\zeta^x(\omega)\,,\\
  & X_t^{x,n}(\omega):=\Theta_t^x(\omega)\,, \qquad \mbox{for } t\in[0,\zeta_1(\omega)]\,,
 \end{align*}
 and
 \begin{align*}
  & \zeta_{i+1}(\omega):=\zeta^{X^{x,n}_{\zeta_i(\omega)}(\omega)}(\omega)+\zeta_i(\omega)\,,\\
  & X_t^{x,n}(\omega):=\Theta_{t-\zeta_i(\omega)}^{X^{x,n}_{\zeta_i(\omega)}(\omega)}(\omega)\,, \qquad \mbox{for } t\in[\zeta_i(\omega),\zeta_{i+1}(\omega)]\,.
 \end{align*}
Then, due to the strong Markov property (cf. \cite[Chapter V.6, Theorem 32]{protter2004}), $X^{x,n}$ stopped at
$\zeta^{\cD_n}:=\inf\{ t \ge 0\vert X_t^{x,n} \notin \cD_n \}$ defines a \usol{} to \eqref{eq:SDE} on $\cD_n$.\\
 
\underline{Step 3:} It remains to extend the solution to the whole domain $\domain$. We already know that for each $x \in \domain$
there is an $n_0 \in \N$ such that for all $n\ge n_0$, $x \in \cD_n$. The sequence of stopping times $\zeta^{\cD_n}=\inf\{ t \ge 0\vert X_t^{x,n} \notin \cD_n \}$
is increasing in $n$ and thus we may define $\hat{\zeta}:=\lim_{n \to \infty} \zeta^{\cD_n}$. From Step 2 we know that for all such $n \ge n_0$ a \usol{}
 \begin{align*}
  X^{x,n}:\Omega \times [0,\zeta^{\cD_n}) \to \domain
 \end{align*}
 exists. Clearly, for $n_1,n_2 \ge 0$, $n_1 \neq n_2$ we have
 \begin{align*}
  X_t^{x,n_1}=X_t^{x,n_2} \qquad \forall t \in [0,\zeta^{\cD_{n_1}} \wedge \zeta^{\cD_{n_2}}] \mbox{ a.s.}
 \end{align*}
 Therefore, define
 \begin{align*}
  X_t^x(\omega):=X_t^{x,n}(\omega) \qquad \mbox{for } t\le \zeta^{\cD_n}\,,
 \end{align*}
 which is our \msol.
 
\end{proof}

We proved existence and uniqueness of a \msol{} of system \eqref{eq:SDE} under Assumptions \ref{ass:terms} and \ref{ass:b1}.
Naturally, by imposing stronger conditions on the coefficients we can show the corresponding global result.

\begin{theorem}\label{thm:global}
Let Assumptions \ref{ass:terms} and \ref{ass:b1} be fulfilled. Additionally, let $\sigma$  be globally Lipschitz and let $\mu$ satisfy a linear growth condition, i.e., $\|\mu(x)\|\le D_1+ D_2\|x\|$ for constants $D_1,D_2>0$.\\
Then there exists a unique global solution to \eqref{eq:SDE}.
\end{theorem}

\begin{proof}
For the \msol{} $X$ we need to show that
\begin{align*}
 \P \left( \limsup_{t \to \zeta} \left\|X_t\right\|^2 = \infty; \, \zeta < \infty \right) = 0
\end{align*}
for all stopping times $\zeta$.
Now, let $\zeta$ be a stopping time with $\P (\zeta < \infty) > 0$.
Furthermore, let $(T_n)_{n \ge 0}$ be the sequence of stopping times defined by
\begin{align*}
T_n: = \inf \left\{ t \ge 0 :  \left\|X_t\right\|^2 \ge n \right\}\,,
\end{align*}
and let
\begin{align*}
X_t^{T_n} & = x_0 + \int_{0}^{t \wedge T_n} \mu \left(X_s^{T_n}\right) \,ds + \int_0^{t \wedge T_n} \sigma \left(X_s^{T_n}\right) \,dW_s \,.
\end{align*}

For $T>0$ large enough such that $\P (\zeta < T) > 0$ we consider
\begin{align*}
&\E \left(\sup_{t \le T} \left\|X_t^{T_n}\right\|^2 \right) \\
&\le   \E \left( \left\|x_0\right\|^2 \right) + \E \left(\sup_{t \le T} \left\| 
\int_0^{t \wedge T_n} \mu \left(X_s^{T_n}\right) \,ds
\right\|^2\right)
 +  \E \left(\sup_{t \le T} \left\|
\int_0^{t \wedge T_n} \sigma \left(X_s^{T_n}\right) \,dW_s
\right\|^2\right)\\
&= :  E_1 + E_2 + E_3.
\end{align*}

Let $D_1,D_2$ be as above.
\begin{align*}
E_2 & \le \E \left( \sup_{t \le T}\, (t \wedge T_n) \int_0^{t \wedge T_n}  \left\|\mu \left(X_s^{T_n}\right)\right\|^2\,ds \right)\\
 & \le \E \left( \sup_{t \le T}\, (t \wedge T_n) \int_0^{t \wedge T_n} \left( D_1+D_2\left\|X_s^{T_n}\right\|\right)^2  \,ds \right)\\
& \le 2 D_1^2 T^2 +2 T D_2^2 \, \int_0^{T}\E \left(\sup_{s \le t}\left\|X_s^{T_n}\right\|^2 \right)\,dt\,.
\end{align*}

By Doob's $\cL^2$-inequality and It\^o's isometry we get
\begin{align*}
E_3 & \le 4 \E \left( \sum_{i=1}^d \left( \int_0^T \sum_{j=1}^d  \sigma_{ij}(X_s^{T_n})\, dW_s^j \right) ^2 \right)
= 4 \sum_{i=1}^d \E \left( \int_0^T \left( \sum_{j=1}^d \sigma_{ij}(X_s^{T_n})\right)^2 \, ds \right)\\
&\le 4d \sum_{i=1}^d \E \left( \int_0^T \sum_{j=1}^d \sigma_{ij}(X_s^{T_n})^2 \, ds \right)
\le C_1 T + C_2 L^2 \int_0^T \E\left( \left\| X_t^{T_n} \right\|^2\right) \, dt\\
&\le C_1 T + C_2 L^2 \int_0^T \E\left(\sup_{s \le t} \left\| X_s^{T_n} \right\|^2\right) \, dt\,,
\end{align*}
where $C_1,C_2$ are constants independent of $n$ and $L$ is the maximum of all appearing Lipschitz constants. Thus, we have
\begin{align*}
E_1 + E_2 + E_3 & \le \left\|x_0\right\|^2 + 2 D_1^2 T^2 + C_1 T +(2 T D_2^2+C_2 L^2) \, \int_0^{T}\E \left(\sup_{s \le t}\left\|X_s^{T_n}\right\|^2 \right)\,dt\\
& =: A (T) + B(T) \int_0^{T}\E \left(\sup_{s \le t}\left\|X_s^{T_n}\right\|^2 \right)\,dt\,.
\end{align*}

Therefore,
\begin{align*}
 \E \left( \sup_{t \le T^*} \left\| X_t^{T_n} \right\|^2 \right) \le A(T) + B(T) \int_0^{T^*} \E \left( \sup_{s \le t} \left\| X_s^{T_n}\right\|^2 \right) \, dt
\end{align*} 
$\forall T^* \le T$. We can directly apply Gronwall's inequality and get
\begin{align*}
 \E \left( \sup_{t \le T^*} \left\| X_t^{T_n} \right\|^2 \right) \le A(T) e^{B(T) T^*} \quad \forall n \mbox{ and } \forall T^* \le T\,.
\end{align*} 
Sending $n \to \infty$ we arrive at
\begin{align*}
 \E \left( \sup_{t \le T} \left\| X_t\right\|^2 \right) \le C(T) < \infty.\,
\end{align*}

Since the above expectation is finite for each $T$, we can conclude
\begin{align*}
 &\P \left( \limsup_{t \to \zeta} \left\|X_t \right\|^2 = \infty; \, \zeta < \infty \right)
 = \lim_{T \to \infty} \P \left( \limsup_{t \to \zeta} \left\|X_t\right\|^2 = \infty; \, \zeta < T \right)\\
 &= \lim_{T \to \infty} \P \left( \limsup_{t \to \zeta} \left\|X_t \right\|^2 = \infty \, \Big\vert \, \zeta < T \right)\P\left(\zeta<T\right)
 =0\,.
\end{align*}
\end{proof}

As a generalization we would like to allow discontinuities not only along
$\{x_1=0\}$, but also along some sufficiently regular hypersurface
$\{x\in\domain:f(x)=0\}$. We consider the system
\begin{align}\label{eq:SDEtransf}
dX_t &= \mu (f(X_t),X^2_t,\ldots,X^d_t) d t + \sigma (X_t) \,dW_t\,,
\end{align}
$X_0=x_0$.

\begin{assumption}\label{ass:termstransf}
We assume the following for the coefficients of \eqref{eq:SDEtransf}.
 \begin{enumerate}
  \item[(i)] $\sigma_{ij} \in C^{1,3}(\R\times \R^{d-1})$, for $i,j=1,\ldots,d$,
  \item[(ii)] $f \in C^{3,5}(\R\times \R^{d-1})$, and $\vert \frac{\partial f}{\partial x_1}\vert > 0$,
  \item[(iii)] for some constant $c$ and for all $x \in \domain$
	     \begin{align*}
              \left\| \nabla f(x) \cdot \sigma(x)  \right\|^2\ge c > 0\,.
             \end{align*}
 \end{enumerate}
\end{assumption}

\begin{remark}
One may notice that item (iii) replaces the item (iii) from Assumption \ref{ass:terms}. Item (iii) has a nice geometric interpretation:
the diffusion component must not be parallel to the surface where the drift is discontinuous.
\end{remark}

Now we have a look at the transformation $U_t=f(X_t)$.
Due to our assumptions and \cite[Theorem 9.24]{rudin1976}, for any $x_0$ there exists a function $e \in C^{3,5}(\R\times \R^{d-1})$ such that
\begin{align*}
e(f(x),x_2,\ldots,x_d)&=x_1\,,\\
f(e(u,x_2,\ldots,x_d),x_2,\ldots,x_d)&=u\,.
\end{align*}
We define
\begin{align*}
\bar{\mu}_1(u,x_2,\ldots,x_d) &= \sum_{i=1}^d \mu_i(u,x_2,\ldots,x_d) \frac{\partial f}{\partial x_i}+\frac{1}{2}\sum_{i,j=1}^d (\sigma \sigma^\top)_{ij} \frac{\partial^2 f}{\partial x_i \partial x_j}\,,\\
\bar{\mu}_i(u,x_2,\ldots,x_d) &= \mu_i(u,x_2,\ldots,x_d)\,, \qquad i=2,\ldots,d\,,\\
\bar{\sigma}_{1j} (u, x_2,\ldots,x_d) &= \sum_{i=1}^d  \sigma_{ij} \frac{\partial f}{\partial x_i}\,, \qquad j=1,\ldots,d\,,\\
\bar{\sigma}_{ij} (u, x_2,\ldots,x_d) &= \sigma_{ij}\,, \qquad i=2,\ldots,d\,;\,\,j=1,\ldots,d\,,
\end{align*}
where all missing arguments are $(e(u,x_2,\ldots,x_d),x_2,\ldots,x_d)$.\\
This leads to the system
\begin{align}\label{eq:newsystemtransf}
(dU_t,dX^2_t,\ldots,dX^d_t)^\top &= \bar{\mu} (U_t,X^2_t,\ldots,X^d_t) d t + \bar{\sigma} (U_t,X^2_t,\ldots,X^d_t) d W_t\,,
\end {align}
$(U_0,X^2_0,\ldots,X^d_0)=(u,x_2,\ldots,x_d)=(f(x),x_2,\ldots,x_d)$.

\begin{theorem}\label{thm:transf}
 Let system \eqref{eq:SDEtransf} fulfill Assumptions \ref{ass:b1} and
\ref{ass:termstransf}. Then \eqref{eq:SDEtransf} has a \msol.  If, in addition,
$\bar\sigma$ is globally Lipschitz and $\bar\mu$ satisfies linear growth,
then \eqref{eq:SDEtransf} has a unique global solution.
\end{theorem}

\begin{proof}
Note that for the transformed system \eqref{eq:newsystemtransf} our original
assumptions on the coefficients are fulfilled.  Therefore, $X=e(U,X^2,\ldots,X^d)$ is the unique local solution to \eqref{eq:newsystemtransf}. Furthermore, we
can apply Theorem \ref{thm:theorem2} and get that system
\eqref{eq:newsystemtransf} has a \msol.  
The remaining assertion follows immediately from Theorem \ref{thm:global}.
\end{proof}

\begin{remark}
While in \citet{veretennikov1984} only those components
of the drift are allowed to be discontinuous, the related part
of the diffusion of which is uniformly elliptic,
our result allows all drift components to be discontinuous along a sufficiently smooth hypersurface.
Furthermore, in \cite{veretennikov1984} all coefficients need to be bounded, which is not required herein.
Instead however, we require more smoothness of the coefficients.\\
The proof of the main result from \cite{veretennikov1984} relies on highly
evolved methods from the literature on PDEs. Concretely, in \cite{veretennikov1984} first weak existence is proven and then pathwise uniqueness.
In contrast to that our proof is more direct.
\end{remark}


\section{Examples}
\label{sec:Example}

At this point we return to the initially cited problem from \cite{sz12}, which originates from stochastic optimal control in mathematical finance.

\begin{example}
 In \cite{sz12} the question arises whether the system
\begin{equation}\label{eq:example}
\begin{aligned}
dX^1_t &= \left(X^2_t - \kappa 1_{\{ X^1_t \geq b (X^2_t)\}} \right) \, dt + \sigma d W_t^1\,,\\
dX^2_t &= \frac{1}{\sigma} (\theta_2 - X^2_t)(X^2_t - \theta_1) d W_t^1\,,
\end{aligned}
\end{equation}
where $\sigma>0$, and $\theta_1, \theta_2$ are constants with $\theta_1 \leq X^2_t \leq \theta_2$ for $t \geq 0$, has a solution.
The motivation of system \eqref{eq:example} is as follows.
$X^1$ originally corresponds to a dividend paying firm value process of a company which is described by a Brownian motion with drift.
Since the drift is assumed to be unobservable, it is replaced by its estimator, which by an application of filtering theory is given by the process $X^2$.
The dividend payments are governed by the threshold function $b$ such that
whenever $X^1_t \ge b(X^2_t)$, dividends are paid at a constant rate
$\kappa$.\\ If $f(x,y)=x-b(y)$ fulfills items (ii) and (iii) of Assumption
\ref{ass:termstransf}, all conditions on the coefficients are fulfilled.
Thus, applying Theorem \ref{thm:transf} implies the existence of a \ugsol{} of \eqref{eq:example}.
\end{example}

Now we complete the counterexample mentioned in the introduction: 
if the crucial condition $(\sigma \sigma^\top)_{11}(x)\ge c>0$ is violated, then there is no solution.

\begin{example}
There is no measurable function $\tilde X$ satisfying
the differential equation
 \begin{align*}
   \tilde X_t&=\int_0^t\left(\frac{1}{2}-\mathrm{sgn}(\tilde X_s)\right) ds\,,\\
   \tilde X_0&=0\,.
 \end{align*}
\end{example}

\begin{proof}
Suppose such an $\tilde X$ would indeed exist.
Since $\mathrm{sgn}(\tilde X_s)$ is bounded, $t\mapsto \int_0^t\mathrm{sgn}(\tilde X_s)\, ds$
is continuous and therefore $\tilde X$ is continuous.

Let $\varepsilon>0$.
$\tilde X$ cannot be positive on $(0,\varepsilon)$, since then the integrand
would equal $-\frac{1}{2}$ on $(0,\varepsilon)$ and thus 
$\tilde X_\varepsilon=\tilde X_t-\int_t^\varepsilon \frac{1}{2} ds
=\tilde X_t-\frac{\varepsilon-t}{2}$. From this we get 
$\tilde X_t=\tilde X_\varepsilon+\frac{\varepsilon-t}{2}>\frac{\varepsilon-t}{2}$, 
thus contradicting the continuity of $\tilde X$ in $t=0$.

Now suppose there was $t_1>0$ with $\tilde X_{t_1}>0$. Then the set 
$A=\{t\in[0,t_1]:\tilde X_t=0\}$ is non-empty and bounded from above. Set $t_0:=\sup A$. 
From the continuity of $\tilde X$ it follows that $\tilde X_{t_0}=0$
and hence $t_0<t_1$.
Thus $\hat X_t:=\tilde X_{t-t_0}$ defines a solution of \eqref{eq:sde_no_sol1} with
$\hat X_0=0$ and $\hat X_t>0$ on $(0,t_1-t_0)$. But we have already ruled out
the existence of such a solution.

Thus $\tilde X$ can never be positive and, by analogous arguments, $\tilde X$ can never be 
negative. But neither can we have for $\varepsilon>0$ that $\tilde X_t=0$ for all 
$t\in (0,\varepsilon)$, since that would imply
$\tilde X_t=\int_0^t(\frac{1}{2}-\mathrm{sgn}(0))\,ds=\frac{t}{2}\ne 0$.
\end{proof}

\appendix

\section{Proof of It\^o's formula}\label{sec:app-ito}
The following well-known auxiliary results are required for the proof of Theorem \ref{thm:ito}. We recall them for the convenience of the reader and to fix notations.

\begin{lemma}\label{lemma:conv-diff}
Let $f,g\colon \R^d\to \R$ be measurable functions with $f$ locally bounded and  
$g$ continuous with  support  contained in $B_r(0)$ for some $r\in (0,\infty)$. Then
\begin{enumerate} 
\item \label{it:cont}$f*g$ is continuous.

\item \label{it:g-diff}If $g$ is continuously partially differentiable in direction $v\in\R^d\setminus\{0\}$, then $f*g$ 
is continuously partially differentiable in direction $v$ and
\[
\frac{\partial}{\partial v} (f*g)=f* \Big(\frac{\partial}{\partial v}g\Big)\,.
\]  

\item \label{it:f-diff}If $U\subseteq \R^d$ is open and $f|_U$ is continuously partially differentiable in direction $v\in\R^d\setminus\{0\}$, then $f*g$ 
is continuously partially differentiable in direction $v$ in all $x\in U$ with $B_{2r}(x)\subseteq U$, and
\[
\frac{\partial}{\partial v} (f*g)= \Big(\frac{\partial}{\partial v}f\Big)*g\,.
\]  

\end{enumerate}
\end{lemma}

\begin{proof}
\ref{it:cont}.
For every $x\in \R^d$ and every sequence $(x_n)_{n\in \N}$ in $B_1(x)$ tending to $x$ we have
\begin{align*}
|(f*g)(x_n)-(f*g)(x)|
&=\Big|\int_{\R^d} f(t)g(x_n-t)dt-\int_{\R^d} f(t)g(x-t)dt\Big|\\
&=\Big|\int_{B_{r+1}(x)} f(t)\big(g(x_n-t)-g(x-t)\big)dt\Big|\\
&\le\big\|f|_{B_{r+1}(x)}\big\|_\infty\int_{B_{r+1}(x)} |g(x_n-t)-g(x-t)|dt.
\end{align*}
Since $g$ is continuous, by bounded convergence the right hand side tends to $0$ as $n\to \infty$. 
Thus $f*g$ is continuous.

\ref{it:g-diff}. For every sequence $(a_n)_{n\in \N}$ of non-zero real numbers in $(-\|v\|^{-1},\|v\|^{-1})$  tending to 0 
define $h\colon \N\times\R^d\to\R$ 
by \[h_n(t):= f(x-t)\frac{1}{a_n}\big(g(t+a_n v)-g(t)\big)\,.\]
For every $n\in \N$, $h_n$ is a measurable function with support contained in $B_{r+1}(0)$ and  
$\lim_{n\to\infty}h_n(t)=f(x-t)\frac{\partial}{\partial v}g(t)$ for all $t\in \R^d$.
Moreover,  since the support of $g$ is contained in  $B_r(0)$ and since $g$ is continuously partially differentiable in direction $v$,  
$|f(x-t)(g(t+a_n v)-g(t))|\le \big\|f|_{B_{r+1}(x)}\big\|_\infty|a_n|\|\frac{\partial}{\partial v}g\|_\infty$ by the mean value theorem.
Thus the sequence $(h_n)_{n\in\N}$ is bounded. We have
\begin{align*}
(f*g)(x+a_n v)-(f*g)(x)
&=\int_{\R^d}f(t)\big(g(x+a_n v-t)-g(x-t)\big)dt \\
&=\int_{B_{r+1}(x)}f(x-t)\big(g(t+a_n v)-g(t)\big)dt \,,
\end{align*}
such that by bounded convergence
\begin{align*}
\lim_{n\to\infty}\frac{1}{a_n}\big((f*g)(x+a_n v)-(f*g)(x)\big)
&=\lim_{n\to\infty}\int_{B_{r+1}(x)}h_n(t)dt
=\int_{B_{r+1}(x)}\lim_{n\to\infty}h_n(t)dt\\
&=\int_{\R^d}f(x-t)\frac{\partial}{\partial v}g(t)dt\,.
\end{align*}
Since the sequence $(a_n)_{n\in\N}$ was arbitrary, we get 
\[
\frac{\partial}{\partial v}(f*g)(x)
=\int_{\R^d}f(x-t)\frac{\partial}{\partial v}g(t)dt=
\Big(f*\big(\frac{\partial}{\partial v}g\big)\Big)(x)\,.
\]
\ref{it:f-diff}. Let $x\in U$ with
 $B_{2r}(x)\subseteq U$.  For an arbitrary sequence $(a_n)_{n\in \N}$ of non-zero real numbers in $(-\|v\|^{-1}r,\|v\|^{-1}r)$ tending to 0 
define $h\colon \N\times\R^d\to\R$ 
by \[h_n(t):= \frac{1}{a_n}\big(f(t+a_n v)-f(t)\big)g(x-t)\,.\]
For every $n\in \N$, $h_n$ is a measurable function with with support contained in $B_{r}(x)$ and  
\[
\lim_{n\to\infty}h_n(t)=
\begin{cases}
\big(\frac{\partial}{\partial v}f(t)\big)g(x-t) & \text{for all } t\in B_r(x)\,,\\
0& \text{for all } t\notin B_r(x)\,.
\end{cases}
\]
Moreover,  since the support of $g$ is contained in  $B_r(0)$,  
$|(f(t+a_n v)-f(t))g(x-t)|\le |a_n|\big\|\frac{\partial}{\partial v}f|_{B_{2r}(x)}\big\|_\infty\big\|g\big\|_\infty$.
Thus the sequence $(h_n)_{n\in\N}$ is bounded and we can finish the proof analog to that of item \ref{it:g-diff}.
\end{proof}

\begin{definition}
Let $k\in\N\cup\{0\}$ and let $(\phi_n)_{n\in \N}$ be a sequence of 
$C^k$ functions $\phi_n\colon \R^d\to [0,\infty)$
with 
\begin{enumerate}
\item for all $n\in \N$ the support of  $\phi_n$ is contained in $B_\frac{1}{n}(0)$;
\item $\int_{\R^d} \phi_n=1$.
\end{enumerate}
Then we call $(\phi_n)_{n\in \N}$ a $C^k$-{\em approximate identity}.
\end{definition}

\begin{lemma}\label{lemma:app-ident-convergence}
Let $f\colon \R^d\to \R$ be locally integrable 
and $(\phi_n)_{n\in \N}$ a $C^k$-approximate 
identity, $k\in\N\cup\{0\}$. Then  for all points of continuity $x$ of $f$, 
\[
\lim_{n\to \infty} (f*\phi_n)(x)=f(x)\,.
\]
If $f$ is continuous with compact support, then 
\[
\lim_{n\to\infty}\|f*\phi_n-f\|_\infty=0\,.
\]
\end{lemma}

\begin{proof}
Let $x$ be a point of continuity of $f$ and let $\varepsilon>0$.
Then there exists $\delta>0$ such that 
for all $t\in \R^d$ with $\|t\|<\delta$ it holds that $|f(x+t)-f(x)|<\varepsilon$. Now let 
$n_0>\frac{1}{\delta}$, such that $\int_{B_\delta(0)}\phi_n(t)dt=1$ for all $n\ge n_0$. We therefore have,
for all $n\ge n_0$, 
\begin{align*}
|(f*\phi_n)(x)-f(x)|
&=\Big|\int_{B_\delta(0)} (f(x-t)-f(x))\phi_n(t)dt\Big|\\
&\le \int_{B_\delta(0)} |f(x-t)-f(x)|\phi_n(t)dt< \varepsilon \int_{B_\delta(0)} \phi_n(t)dt
= \varepsilon\,,
\end{align*}
that is, \( \lim_{n\to\infty}|(f*\phi_n)(x)-f(x)|=0\,.
\)
 
If $f$ has compact support, then $f$ is uniformly continuous and 
$\delta$ in the earlier argument can be chosen independently of $x$.
\end{proof}

\begin{lemma}\label{lemma:app-ident-bounded}
Let $(\phi_n)_{n\in \N}$ be a $C^k$-approximate 
identity, $k\in\N\cup\{0\}$.
If $f\colon \R^d\to \R$ is measurable and locally bounded, then
for all $x\in\R^d$,
\[
\limsup_{n\to \infty} |(f*\phi_n)(x)|
\le \lim_{n\to\infty}\sup\{|f(y)|\colon y\in B_\frac{1}{n}(x)\} \;.
\]
In particular, if $f$ is measurable and bounded, then
for all $x\in\R^d$,
\[
\limsup_{n\to \infty} |(f*\phi_n)(x)|
\le \|f\|_\infty<\infty \;.
\]
\end{lemma}

\begin{proof}
Let $x\in \R^d$. Then 
\begin{align*}
|(f*\phi_n)(x)|&\le\int_{\R^d}|f(x-t)|\phi_n(t)dt=\int_{B_{\frac{1}{n}}\!(0)}|f(x-t)|\phi_n(t)dt\\
&\le\sup\{|f(y)|\colon y\in B_\frac{1}{n}(x)\}\,,
\end{align*}
from which the first claim follows.
If in addition  $f$ is bounded,  $|(f*\phi_n)(x)|\le\sup\{|f(y)|\colon y\in B_\frac{1}{n}(x)\}\le \|f\|_\infty$.
\end{proof}
\begin{proof}[Proof of It\^o's formula (Theorem \ref{thm:ito})]
W.l.o.g.~we may assume that $f$ has compact support, such that by our assumptions 
$f$ and all 
its first and second partial derivatives (where they exist) are bounded by some
constant $K$.

Let $(\phi_n)_{n\in\N}$ be a $C^2$-approximate identity and let,
 for all $n\in \N$, $f_n:=f*\phi_n$, such that, by Lemma \ref{lemma:conv-diff} item \ref{it:g-diff}, $f_n$ is a 
$C^2$-function with 
$\lim_{n\rightarrow \infty} \|D(f-f_n)\|_\infty=0$ for every
$D\in \{1,\frac{\partial}{\partial x_i},\frac{\partial^2}{\partial x_i\partial x_j}\big| 1\le i\ne j\le d\}$ by Lemma \ref{lemma:app-ident-convergence}.  
For all $i\in\{1,\ldots,d\}$ we have $\frac{\partial^2 }{\partial x_i^2}f_n (x)=((\frac{\partial^2 }{\partial x_i^2}f)*\phi_n)(x)$
for all $x\notin\Theta^i$ and $n$ large enough, by Lemma \ref{lemma:conv-diff} item \ref{it:f-diff},   and thus $\big|\frac{\partial^2 }{\partial x_i^2}f_n (x)\big|\le K$ by Lemma \ref{lemma:app-ident-bounded}. 
Since $\frac{\partial^2 }{\partial x_i^2}f_n$ is continuous,  
$\big\|\frac{\partial^2 }{\partial x_i^2}f_n\big\|_\infty\le K$. We have obtained
\begin{equation}\label{eq:unif-bd}
\forall x\in\R^d\;\forall n\in\N\colon \Big|\frac{\partial^2 }{\partial x_i^2}f_n(x)\Big|\le K\,.
\end{equation}
Since $f_n \in C^2$, It\^o's formula holds:
\begin{align}\label{eq:ito-help}
f_n(X_t)&=f_n(X_0)+\sum_{i=1}^d \int_0^t \frac{\partial }{\partial x_i}f_n(X_s)\,dX_s^i
+\frac{1}{2} \sum_{i,j=1}^d \int_0^t \frac{\partial^2 }{\partial x_i\partial x_j}f_n(X_s)\,d[X^i,X^j]_s\,.
\end{align}
By uniform convergence of the integrands, we have convergence
\begin{align*}
\lim_{n\to \infty}\int_0^. \frac{\partial }{\partial x_i}f_n(X_s)\,dX_s^i
&= \int_0^. \frac{\partial }{\partial x_i}f(X_s)\,dX_s^i\\
\lim_{n\to \infty}\int_0^. \frac{\partial^2 }{\partial x_i\partial x_j}f_n(X_s)\,d[X^i,X^j]_s
&= \int_0^. \frac{\partial^2 }{\partial x_i\partial x_j}f(X_s)\,d[X^i,X^j]_s
\end{align*}
u.c.p.~for all $i,j\in \{1,\ldots,n\}$ with $i\ne j$.

Let us consider the term with the second derivative w.r.t.~$x_i$. On $\Theta^i$ this derivative is not defined, so set $\frac{\partial^2 f}{\partial x_i^2}\equiv 0$ on $\Theta^i$, for definiteness.
For every 
$t$,
\begin{align*}
\int_0^t\frac{\partial^2 }{\partial x_i^2}f_n(X_s)\,d[X^i]_s
&=
\int_0^t1_{\{|X^i_s|>\frac{1}{n}\}}\frac{\partial^2 }{\partial x_i^2}f_n(X_s)\,d[X^i]_s
+\int_0^t1_{\{|X^i_s| \le \frac{1}{n}\}}\frac{\partial^2 }{\partial x_i^2}f_n(X_s)\,d[X^i]_s\,.
\end{align*}
By Fatou's lemma, we have for every sequence $(h_n)_{n\in\N}$ of non-negative, measurable, and bounded functions
with $\|h_n\|_\infty\le 2K$ for all $n\in\N$,
\begin{align*}
\limsup_{n\to \infty}\int_0^t1_{\{|X^i_s|\le\frac{1}{n}\}}h_n(X_s)\,d[X^i]_s
&\le\int_0^t\limsup_{n\to \infty}1_{\{|X^i_s|\le\frac{1}{n}\}}h_n(X_s)\,d[X^i]_s\\
&\le\int_0^t1_{\{|X^i_s|=0\}}\limsup_{n\to \infty}h_n(X_s)\,d[X^i]_s\\
&\le 2K\int_0^t1_{\{X^i_s= 0\}}\,d[X^i]_s\;.
\end{align*}
From \cite[Chapter 3, Theorem 7.1]{karatzas1991} we know that 
\[
\int_0^t 1_{\{X^i_s=0\}} \,d[X^i]_s=2\int_\R 1_{\{0\}}(a)\Lambda_t(a) \,da=0\,,
\]
where $\Lambda_t(a)$ denotes the local time of $X^i$ in $a$ up to time $t$.
Therefore 
\begin{align*}
\limsup_{n\to \infty}\int_0^t1_{\{|X^i_s|\le\frac{1}{n}\}}h_n(X_s)\,d[X^i]_s&=0\;.
\end{align*}
In particular this holds for $h_n=\Big|\frac{\partial^2 }{\partial x_i^2}f_n-\frac{\partial^2 }{\partial x_i^2}f\Big|$, which by \eqref{eq:unif-bd} is uniformly bounded by $2K$ in $x$ and $n$, such that 
\begin{equation}\label{eq:zero}
\limsup_{n\to \infty}\int_0^t1_{\{|X^i_s|\le\frac{1}{n}\}}\Big|\frac{\partial^2 }{\partial x_i^2}f_n(X_s)-\frac{\partial^2 }{\partial x_i^2}f(X_s)\Big|\,d[X^i]_s=0\;.
\end{equation}
By bounded convergence, we also have 
\begin{equation}\label{eq:nonzero}
\lim_{n\to \infty}\int_0^t1_{\{|X^i_s|>\frac{1}{n}\}}\Big(\frac{\partial^2 }{\partial x_i^2}f_n(X_s)-\frac{\partial^2 }{\partial x_i^2}f(X_s)\Big)\,d[X^i]_s=0\,.
\end{equation}
Combining \eqref{eq:zero} and \eqref{eq:nonzero} gives
\begin{align*}
\lim_{n\to \infty}\Big|\int_0^t&\frac{\partial^2 }{\partial x_i^2}f_n(X_s)\,d[X^i]_s-\int_0^t\frac{\partial^2 }{\partial x_i^2}f(X_s)\,d[X^i]_s\Big|\\
&\le 
\lim_{n\to \infty}\Big|\int_0^t1_{\{|X^i_s|>\frac{1}{n}\}}\Big(\frac{\partial^2 }{\partial x_i^2}f_n(X_s)-\frac{\partial^2 }{\partial x_i^2}f(X_s)\Big)\,d[X^i]_s\Big|\\
&\quad +
\lim_{n\to \infty}\int_0^t1_{\{|X^i_s|\le\frac{1}{n}\}}\Big|\frac{\partial^2 }{\partial x_i^2}f_n(X_s)-\frac{\partial^2 }{\partial x_i^2}f(X_s)\Big|\,d[X^i]_s=0\,.
\end{align*}

Thus, for $n\rightarrow \infty$ all terms in \eqref{eq:ito-help} converge to the corresponding term with $f_n$ replaced by $f$.
\end {proof}

%
%
%
%
%
%

%

\vspace{2em}
\noindent{\bf Acknowledgments.} The authors thank the anonymous referees for their comments which helped to 
greatly improve the paper. In addition we thank Thomas Müller-Gronbach and Wolfgang Stockinger for pointing
out a mistake in an earlier version.


\end{document}